\documentclass[a4paper,10pt]{article}

\usepackage{bm}
\usepackage{amsmath}
\usepackage{amssymb}
\usepackage{amsthm}
\usepackage{mathrsfs}
\usepackage{color}
\usepackage[dvips]{graphicx}
\usepackage[all]{xy}
\usepackage{array, booktabs}

\theoremstyle{plain}
\newtheorem{them}{Theorem}[section]
\newtheorem{lemma}[them]{Lemma}
\newtheorem{prop}[them]{Proposition}
\newtheorem{coro}[them]{Corollary}
\theoremstyle{definition}
\newtheorem{defi}[them]{Definition}
\newtheorem{exam}[them]{Example}

\newtheorem{conj}[them]{Conjecture}

\newcommand{\Li}{\mathop{\mathrm{Li}}}
\newcommand{\wchi}{\widetilde{\chi}}

\begin{document}

\title{The distribution of primes and Euler characteristic}
\author{Kazunori Noguchi \thanks{noguchi@cc.kogakuin.ac.jp}}
\date{}
\maketitle
\begin{abstract}
We study the distribution of primes from a topological viewpoint. Certain conjecture is introduced, and we show that it is equivalent to the Riemann Hypothesis.
\end{abstract}


\thispagestyle{empty}

\section{Introduction}
\subsection{Historical background}
Prime numbers provide us many attractive problems. ``Attractive" means very surprising and extremely difficult.

One of such problems is to estimate the prime counting function $\pi (x)$, the number of primes not exceeding a real number $x$. This function is in a central place in number theory, and it irregularly increases as $x\to \infty$. However, the Prime Number Theorem states that, surprisingly, the elementary function $\frac{x}{\log x}$ approximates to $\pi(x)$: $$ \lim_{x\to \infty}  \frac{\pi(x)}{  x/\log x   } =1.$$ (Since the logarithmic integral function $\Li x = \int^x_2 \frac{dt}{\log t}$ is almost the same as $\frac{x}{\log x}$, we can replace the denominator by $\Li x$.)

The first proof of the fact was given by Hadamard and de la Vall\'ee Poussin independently, and they used the Riemann zeta function $\zeta(z)$. See, for example, Chapter III of \cite{Tit86}. The rest of the problem is to estimate the error term $| \pi(x)  - \Li x|$. The best result is $$ \pi(x) =\Li x +O\left(   x\exp \left(  -C \frac{ (\log x)^{  \frac{3}{5}   }    }{ (\log \log x)^{  \frac{1}{5}  }   } \right) \right) $$ for some constant $C>0$. See (12.27) of \cite{Ivic}. However, the truth of the Riemann Hypothesis is equivalent to the following very sharper estimation:
$$ \pi(x)= \Li x +O\left(  x^{ \frac{1}{2}  }  \log x\right) .$$The best result above has not been improved for a long time.

The Prime Number Theorem and the Riemann Hypothesis have many equivalent propositions, and Bj\"orner gave a topological interpretation to some of them \cite{Bjo11}.

For a squarefree positive integer $k$, let $P(k)$ be the set of prime factors of $k$. For any $n\ge 1$, define an abstract simplicial complex $\Delta_n$ to be the set of $P(k)$ for all squarefree integers $1\le k\le n.$  Let $X$ be a finite set. An \textit{abstract simplicial complex} $\Delta$ is a family of subsets of $X$ that is closed under taking subsets. Namely, if $\sigma$ belongs to $\Delta$ and $\tau$ is a subset of $\sigma$, then $\tau$ also does. We are allowed to take an empty set as $\tau$. If the cardinality of $\sigma$ is $d+1$, then we call $\sigma$ a \textit{$d$-simplex}. The \textit{dimension} of $\Delta$ is the maximum integer $d$ such that there exists a $d$-simplex of $\Delta$. Note that one is squarefree; therefore, all $\Delta_n$ contain $P(1)=\emptyset$. For example, $$ \Delta_6=\left\{ \emptyset, \{2\}, \{3\} , \{5\},\{2,3\}   \right\} ,$$ and it is visualized  as follows. $$\xymatrix{ 2\ar@{-}[r] &3&5}$$ Furthermore, $\Delta_{15}$ is $$\xymatrix{ 7\ar@{-}[d] &3&& \\2\ar@{-}[r] \ar@{-}[ur]&5\ar@{-}[u]&11&13, }$$ and we can find the genus by $\{2,3\}, \{2,5\}$, and $\{3,5\}$. It is very important if a space has a genus or not in topology. However, if $n=30$, then the genus is filled by the triangle $\{2,3,5\}$. As $n$ grows, $d$-dimensional genera are born and filled. The behavior is very complicated. 

Bj\"orner gave the following equivalences: $$ \text{Prime Number Theorem} \Leftrightarrow\widetilde{\chi}(\Delta_n) =o(n)$$ $$  \text{Riemann Hypothesis} \Leftrightarrow \widetilde{\chi}(\Delta_n) =O(n^{\frac{1}{2} +\varepsilon}).$$The Euler characteristic $\widetilde{\chi}(\Delta)$ of an abstract simplicial complex $\Delta$ is given by $\widetilde{\chi}(\Delta)=\sum_{\sigma \in \Delta} (-1)^{\# \sigma -1}$. In fact, $\widetilde{\chi}(\Delta_n)$ is almost the Mertens function $M(n)$; that is, $-M(n)=\widetilde{\chi}(\Delta_n)$. The function $M(x)$ is defined by $\sum_{n\le x} \mu (n)$, where $\mu(n)$ is the classical M\"obius function. If $n=p_1 p_2\dots p_k$, $p_i\not = p_j$, then $\mu(n)=(-1)^k$ and $(-1)^{\# P(n)  -1  }  =(-1)^{k-1}$. Since $\mu(1)=1$ and $(-1)^{ \# \emptyset -1 }=-1$, the equality follows. Theorem 4.14 and 4.15 of \cite{Apo} and Theorem 14.25 (C) imply the two equivalences.

In addition, the following estimation $$  \int^X_2 \left(  \frac{\widetilde{\chi}(\Delta_{[x]})}{x}  \right)^2 dx = O(\log X) $$ implies the simplicity of the zeros of the Riemann zeta function by Theorem 14.29(A) of \cite{Tit86}. The simplicity and the Riemann Hypothesis are major problems in number theory. Hence, to estimate $\widetilde{\chi}(\Delta_n)$ is a very important problem.

\begin{exam}
If $n$ is small, it is easy to compute $\widetilde{\chi}(\Delta_n)$.

\begin{center}
\begin{tabular}{ccccccccccccccccc} \toprule
$n$&1&2&3&4&5&6&7&8&9&10  &11&12&13&14&15&16 \\ \midrule
$\widetilde{\chi}(\Delta_n)$&$-1$&0&1&1&2&1&2&2&2&1&2&2&3&2&1&1 \\ \bottomrule
\end{tabular}

\begin{tabular}{ccccccccccccccc} \toprule
$n$&17&18&19&20&21&22&23&24&25&26&27&28&29&30 \\ \midrule
$\widetilde{\chi}(\Delta_n)$ &2&2&3&3&2&1&2&2&2&1&1&1&2&3  \\ \bottomrule
\end{tabular}

\begin{tabular}{ccccccccccccccc} \toprule
$n$&31&32&33&34&35&36&37&38&39&40&41&42&43&44  \\ \midrule
$\widetilde{\chi}(\Delta_n)$&4&4&3&2&1&1&2&1&0&0&1&2&3&3 \\ \bottomrule

\end{tabular}

\end{center}

\end{exam}

We can observe the oscillation of  $\widetilde{\chi}(\Delta_n) $. The Euler characteristic $\widetilde{\chi}(\Delta_n)$ is not always nonnegative for $n\ge 2$. Indeed, $\widetilde{\chi}(P_{94})=-1$ and 94 is the smallest integer  $n$ such that $n\ge 2$ and $\widetilde{\chi} (\Delta_n)$ is negative.

At first glance, $|\widetilde{\chi}(\Delta_n)|$ is very smaller than $n$, compare to the Prime Number Theorem, in the table, however, the oscillation is very complicated if $n$ is large. Although Mertens conjectured $$ |M(x)| \le \sqrt{x} $$ for $x>1$, Odlyzko and Riele disproved it \cite{OR85}. Namely, the inequality is violated infinitely many times. They showed the existence of counter examples to the conjecture, but no examples have been found concretely.

The first step to study $\widetilde{\chi}(\Delta_n)$ should be to study the homology groups of $\Delta_n$. Bj\"orner tried it, but he found that $\Delta_n$ has the homotopy type of a wedge of spheres. Namely, the homology groups of $\Delta_n$ are almost trivial. He concluded that ``perhaps a study of deeper topological invariants of $\Delta_n$ could add something of value''.

\subsection{Strategy}

Our strategy starts from the $h$-polynomials of $\Delta_n$.

Let $f_i^{\Delta}$ be the number of $i$-simplices of an abstract simplicial complex $\Delta$. For example, $f^{\Delta_6}_{-1}=1, f^{\Delta_6}_{0}=3,$ and $f^{\Delta_6}_1=1.$ Define the \textit{h-polynomial} $h^{\Delta}(z)$ of $\Delta$ by 
\begin{align*}
h^{\Delta}(z):=&\sum_{i=-1}^d f_i^{\Delta} (z-1)^{d-i}\\
=&(-1)^d \wchi(\Delta) +\cdots+z^{d+1},
\end{align*}
where $d=\dim \Delta$. It is easy to show that this polynomial is monic and the constant term is $(-1)^d \wchi(\Delta)$.

If $d_n =\dim \Delta_n$ and $$h^{\Delta_n}(z) = \prod^{d_n +1}_{i=1} \left(z- \rho_i^{\Delta_n} \right),$$ then we have $$ \wchi(\Delta_n)=-\prod^{d_n +1}_{i=1} \rho_i^{\Delta_n}.$$ Hence, we have to estimate the moduli of $\rho_i^{\Delta_n}$, however, it is very hard. Barycentric subdivision is useful in this situation. \textit{Barycentric subdivision} is an operation to produce a new abstract simplicial complex $\Delta'$ from $\Delta$. Precisely, see \S \ref{vari}. The important properties of barycentric subdivision is to preserve Euler characteristic and dimension: $$\wchi(\Delta) =\wchi(\Delta'),  \dim (\Delta) =\dim (\Delta') .$$ The limiting behavior of the zeros of $h^{\Delta}(z)$ under iterated barycentric subdivision is interesting.

\begin{them}\label{main}
Let $\Delta$ be an abstract simplicial complex of dimension $d\ge 1$. Suppose that $\rho^{\Delta^{(k)}}_{\infty} ,\rho^{\Delta^{(k)}}_{0},\rho^{\Delta^{(k)}}_{1},\dots,\rho^{\Delta^{(k)}}_{d-1}$ are the zeros of the $h$-polynomial of the $k$-times subdivided complex $\Delta^{(k)}$ and $\rho^{\Delta^{(k)}}_{\infty}$ and $\rho^{\Delta^{(k)}}_{0}$ have the maximum and minimum moduli among them. As $k\to \infty$, $\rho^{\Delta^{(k)}}_{0}$ tends to zero and $\rho^{\Delta^{(k)}}_{i}$ converge. If $k$ is sufficiently large, $\rho^{\Delta^{(k)}}_{\infty}$ is real and $\left| \rho^{\Delta^{(k)}}_{\infty} \right|$ diverges to $\infty$. Moreover, $$ H_{1,d} \lim_{k\to \infty} \prod_{i=0}^{d-1} \left(z-\rho^{\Delta^{(k)}}_{i}  \right)$$ is the $H$-polynomial $H_d(z)$, and $\prod_{i=1}^{d-1} \rho^{\Delta^{(k)}}_{i}$ converges to $(-1)^{d-1}$.
\end{them}
The numbers $H_{1,d}$ and the polynomials $H_d (z)$ are defined in \S \ref{vari}.

Note that this theorem depends only on the dimension of $\Delta$ rather than $\Delta$ itself. The theorem improves Theorem 3 of \cite{BW08} and Theorem A of \cite{DPS12}, and it implies $$ \wchi(\Delta_n) =\wchi\left(\Delta_n^{(k_n)} \right)=O\left(  \rho^{\Delta^{(k_n)}}_{\infty} \rho^{\Delta^{(k_n)}}_{0} \right),$$ where $k_n$ is sufficiently large. The zero $\rho^{\Delta^{(k_n)}}_{\infty} $ can be estimated by the following theorem:

\begin{them}\label{main2}
Under the same assumption of Theorem \ref{main}, we have $$\rho^{\Delta^{(k)}}_{\infty} \sim - H_{1,d} f_d^{\Delta} (d+1)!^k$$ and $$  |\wchi(\Delta)|\sim H_{1,d} f^{\Delta}_d \left| \rho^{\Delta^{(k)}}_{0} \right|  (d+1)!^k.$$
\end{them}
For two functions $f(x)$ and $g(x)$, define $f(x)\sim g(x)$ if $$\lim_{x\to \infty}  \frac{  f(x) }{  g(x)  } =1.$$

Hence, we have $$\wchi(\Delta_n) =O\left( H_{1,d_n} f_{d_n}^{\Delta_n} \left| \rho^{\Delta_n^{(k_n)}}_{0}  \right| (d+1)!^{k_n} \right).$$Now, we have to estimate the growth of the four factors $d_n=\dim \Delta_n$, $H_{1,d_n}$, $f_{d_n}^{\Delta_n}$, and $\left| \rho^{\Delta_n^{(k_n)}}_{0}  \right|$.

We completely solve the first three problems (Proposition \ref{dim}, Proposition \ref{N3}, and Proposition \ref{N2}), however, the fourth problem remains. In other words, it is the hardest part and the special zeros $\rho^{\Delta^{(k_n)}}_{0}$ have much information to the distribution of primes. Unfortunately, I can not give any way to estimate the zeros. I guess that the zeros are deeply connected to something topological.

At least, I can give the following conjecture. By theorem \ref{main2}, we can put $$\rho^{\Delta^{(k_n)}}_{0} \sim \frac{1}{ (d_n+1)!^{k_n} }  \alpha_n  $$ for some constant $\alpha_n$.

\begin{conj}
\begin{enumerate}
\item $\alpha_n=O\left(  (d_n+1)!^{2} \right)$ \\
\item $\alpha_n=O\left(  (d_n+1)!^{\frac{3}{2} +\varepsilon} \right)$ for any $\varepsilon>0$.
\end{enumerate}
\end{conj}

\begin{exam}
We compute $\alpha_n$ when $n$ is small. 

\begin{center}

\begin{tabular}{ccccccccccccccc} \toprule
$n$&6&7&10&11&13&14&15&17&19&21&22&23&26&29 \\ \midrule
$\alpha_n$&1&2 &$\frac{1}{2}$&1&$\frac{3}{2}$&1&$\frac{1}{3}$&$\frac{2}{3}$& 1&$\frac{1}{2}$& $\frac{1}{5}$& $\frac{2}{5}$& $\frac{1}{6}$& $\frac{1}{3}$ \\ \bottomrule 
\end{tabular}

\begin{tabular}{c|cccccccccccccc} \toprule
$n$&30&31&33&34&35&37&38&39&41&42&43&46&47&51     \\ \midrule
$\alpha_n$&6&8  &6&4&2 &4&2&0&2&2&3&2&3&2   \\ \bottomrule 
\end{tabular}

\begin{tabular}{ccccccccccccccc} \toprule
$n$&53&55&57&58&59&61&62&65&66&67&69&70&71&73 \\ \midrule
$\alpha_n$&3&2&1&0&1&2&1&0&$\frac{2}{3}$ &$\frac{4}{3}$&$\frac{2}{3}$ &1&$\frac{3}{2}$& 2 \\ \bottomrule 
\end{tabular}

\begin{tabular}{cccc|cccccccc} \toprule
$n$&  74&$\cdots$&209&210&211&213&214&215&217&218&219  \\ \midrule
$\alpha_n$&$\frac{3}{2}$&$\cdots$& $\frac{  4  }{  19   }$ &$\frac{  11 }{  2   }$&11&$\frac{  11 }{  2   }$&0 &$-\frac{  11 }{  2   }$&$-11$&$-\frac{33}{2}$&$-22$   \\ \bottomrule 
\end{tabular}

\end{center}

\end{exam}

The first conjecture implies the Prime Number Theorem.
\begin{prop}\label{P1}
If $\alpha_n=O\left(  (d_n+1)!^{2} \right)$, then we have $$   \widetilde{\chi}(\Delta_n) =O\left(  n \exp \left(  -A\frac{\log n \log \log \log n}{\log \log n} \right)  \right)  $$ for some constant $A>0$.
\end{prop}
The result improves the best result of $M(x)$:  $$  M(x)=O \left( x\exp \left( -B \log^{\frac{3}{5}} x (\log \log x)^{-\frac{1}{5}} \right)  \right) . $$
See, for example, Theorem 12.7 of \cite{Ivic}.

The second conjecture is equivalent to the Riemann Hypothesis.

\begin{prop}\label{RHE}
The Riemann Hypothesis is equivalent to $$\alpha_n=O\left(  (d_n+1)!^{\frac{3}{2} +\varepsilon} \right).$$
\end{prop}

\section{Zeros of $h$-polynomials and barycentric subdivision}

In this section, we study combinatorial properties of abstract simplicial complexes under iterated barycentric subdivision, and we prove Theorem \ref{main} and \ref{main2}. The references \cite{BW08} and \cite{DPS12} are earlier literature, and \cite{Koz08}, \cite{Spa66}, and \cite{Sta97} are basic to read this section.

Throughout this section, $\Delta$ is an abstract simplicial complex of dimension $d$.

\subsection{Various numbers and polynomials}\label{vari}

We introduce various numbers and polynomials. They are important to study combinatorial properties of barycentric subdivision.

The \textit{barycentric subdivision} of $\Delta$, denote $\Delta'$, is  defined by the set of an empty set and sets of proper increasing sequences of nonempty simplices of $\Delta$:
$$\Delta'=\{  \emptyset  \}  \cup \{ \{ \sigma_0,\sigma_1,\dots,\sigma_n\}  \mid n\ge 0, \sigma_i \in \Delta, \sigma_i \subsetneq \sigma_{i+1} , \sigma_i \not = \emptyset  \}$$

\begin{exam}
We have seen $$ \Delta_6=\left\{ \emptyset, \{2\}, \{3\} , \{5\},\{2,3\}   \right\} .$$ Let $\sigma_1=\{2\}, \sigma_2=\{3\}, \sigma_3=\{5\},$ and $\sigma_4=\{2,3\}$. Then, $\Delta_6'$ is 
$$ \Delta_6'=\left\{ \emptyset, \{\sigma_1\}, \{\sigma_2\} , \{\sigma_3\}, \{ \sigma_4\} , \{\sigma_1,\sigma_4\} ,\{\sigma_2,\sigma_4\}   \right\} ,$$ and it is visualized as follows.
$$
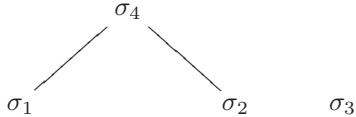
\xymatrix{ &\sigma_4&& \\  \sigma_1 \ar@{-}[ur] &&\sigma_2 \ar@{-}[ul] & \sigma_3}$$
\end{exam}

For $i,d\ge  0$ and a $d$-simplex $\sigma$ of $\Delta$, define $f_{i,d}$ to be the number of elements $\{ \sigma_0,\sigma_1,\dots,\sigma_i\}$ of $\Delta'$ such that $\sigma_i=\sigma$. The numbers depend on $i$ and $d$ rather than $\sigma$ itself. Define $f_{-1,-1}=1$ and $f_{-1,d}=0$ if $d\ge 0$.

By combinatorial observation, the numbers are given by $$ f_{i,d} = (i+1)!S(d+1,i+1),$$
where $S(j,k)$ is the Stirling number of the second kind. See Lemma 1 of \cite{BW08}. It is easy to show that
\begin{align}
 f_{i,d} &= \sum^d_{j=i} \binom{d+1}{j}  f_{i-1,j-1} \label{ind1}
 \end{align}
for any $i,d\ge 0$; therefore, we can compute the numbers inductively.
\begin{center}
\begin{tabular}{lrrrrrrrrr} \toprule
$f_{i,d}$&$d=-1$&$d=0$&$d=1$&$d=2$&$d=3$&$d=4$&$d=5$&$d=6$&$d=7$ \\ \midrule
$i=-1$ &1&0&0&0&0&0&0&0&0 \\
$i=0$ &0&1&1&1&1&1&1&1&1 \\
$i=1$&0&0&2&6&14&30&62&126&254 \\
$i=2$&0&0&0&6&36&150&540&1806&5796 \\
$i=3$&0&0&0&0&24&240&1560&8400&40824 \\
$i=4$&0&0&0&0&0&120&1800&16800&126000 \\
$i=5$&0&0&0&0&0&0&720&15120&191520 \\
$i=6$&0& 0&0&0&0&0&0&5040&141120 \\
$i=7$ &0& 0&0&0&0&0&0&0&40320 \\ \bottomrule
\end{tabular}
\end{center}

Define the \textit{f-polynomial} $f^{\Delta}(z)$ of $\Delta$ by $$ f^{\Delta}(z) := \sum^d_{i=-1} f^{\Delta}_i z^{d-i},$$
where $f^{\Delta}_i$ is the number of $i$-simplices of $\Delta$.

Define the \textit{h-polynomial} $h^{\Delta}(z)$ of $\Delta$ and numbers $h_i^{\Delta}$ by $$ h^{\Delta}(z):=f^{\Delta}(z-1)=\sum_{i=0}^{d+1} h^{\Delta}_iz^{d+1-i}.$$
For example, $$f^{\Delta_6}(z)=z^2+3z+1, h^{\Delta_6}(z)=z^2+z-1.$$
Hence, $h^{\Delta_6}_0=1, h^{\Delta_6}_1=1$, and $h^{\Delta_6}_2=-1$. Furthermore, $$f^{\Delta_{30}} (z)=z^3+10z^2+7z+1, h^{\Delta_{30}}(z)=z^3+7z^2-10z+3.$$ Hence, $h^{\Delta_{30}}_0=1$, $h^{\Delta_{30}}_1=7$, $h^{\Delta_{30}}_2=-10$, and $h^{\Delta_{30}}_3=3$.

We define the next numbers. For $0\le i<d$, define a rational number $F_{i,d}$ by

\begin{multline}F_{i,d} = \sum_{  i_0=i<i_1<\cdots <i_{\ell}<d, \ell \ge 0 }   \frac{  f_{i_0,i_1} }{ (d+1)!-(i_0 +1)!  } \frac{  f_{i_1,i_2} }{ (d+1)!-(i_1 +1)!  } \cdots \\ \cdots\frac{  f_{i_{\ell}, d} }{ (d+1)!-(i_{\ell} +1)!  }  .\end{multline}
Put $F_{d,d}=1$ for $d\ge -1$ and $F_{-1,d}=0$ if $d\ge 0$. At first glance, the numbers seem to be complicated, so some readers might wonder what the numbers are. However, we will see that the column vector ${}^t (F_{-1,d},  F_{0,d} ,\dots,F_{d,d})$ is an eigenvector for $(d+1)!$ of the matrix  $\bm{F}_d$ (Lemma \ref{sym1}). We can inductively compute $F_{i,d}$ by the following: 
 \begin{align}F_{i,d} &=\frac{1}{  (d+1)!-(i+1)!  }  \sum^d_{j=i+1}  f_{i,j} F_{j,d} \label{ind2}\end{align}
 for $-1\le i\le d-1$.  

\begin{center}
\begin{tabular}{cccccccccc} \toprule
$F_{i,d}$ &$d=-1$& $d=0$& $d=1$& $d=2$& $d=3$& $d=4$& $d=5$& $d=6$& $d=7$ \\ \midrule  \addlinespace[3pt]
$i=-1$&1&0&0&0&0&0&0&0&0\\  \addlinespace[3pt]
$i=0$&&1&1&$\frac{1}{2}$ & $\frac{2}{11}$&$\frac{1}{19}$& $\frac{132}{10411} $& $\frac{90}{34399}$ & $\frac{15984}{33846961}$\\  \addlinespace[3pt]
$i=1$&&&1&$\frac{3}{2} $&$\frac{13}{11}$ & $\frac{25}{38}$&$\frac{3004}{10411}$& $\frac{3626}{34399}$& $\frac{12351860}{372316571}$\\  \addlinespace[3pt]
$i=2$&&&&1&2&$\frac{40}{19}$& $\frac{45}{29}$& $\frac{61607}{68798}$&$\frac{7924}{18469}$ \\  \addlinespace[3pt]
$i=3$&&&&&1&$\frac{5}{2}$&$\frac{95}{29}$&$\frac{245}{82}$& $\frac{39221}{18469}$ \\  \addlinespace[3pt]
$i=4$&&&&&&1&3&$\frac{385}{82}$& $\frac{56}{11}$ \\  \addlinespace[3pt]
$i=5$&&&&&&&1&$\frac{7}{2}$& $\frac{70}{11}$ \\  \addlinespace[3pt]
$i=6$&&&&&&&&1&4 \\  \addlinespace[3pt]
$i=7$&&&&&&&&&1 \\ \bottomrule
\end{tabular}
\end{center}
We can find the same table, but bigger than ours, in \S 6  of \cite{DPS12}.

 Define the \textit{$F$-polynomial} $F_d(z)$ of degree $d\ge 0$ by $$ F_d(z) :=\sum^d_{i=-1} F_{i,d} z^{d-i}.$$

Finally, we define the most important numbers and polynomial in this section. 

\begin{defi}
For $d \ge0$, define the \textit{$H$-polynomial} $H_d(z)$ of degree $d$  and numbers $H_{i,d}$, $0\le i \le d+1$, by $$ H_d(z):= F_d(z-1)=\sum^{d+1}_{i=0} H_{i,d} z^{d+1-i}.$$
\end{defi}

This polynomial is the heart of the proof of Theorem \ref{main}. We will find that zeros of $H_{ d}(z)$ directly influences to those of $ h^{ \Delta^{(k)} } (z)$. By the table above, we obtain a new table.
\begin{center}
\begin{tabular}{ccccccccc} \toprule
$H_{i,d}$ & $d=0$& $d=1$& $d=2$& $d=3$& $d=4$& $d=5$& $d=6$& $d=7$ \\ \midrule \addlinespace[3pt]
$i=0$&0&0&0&0&0&0&0&0 \\ \addlinespace[3pt]
$i=1$&0&1&$\frac{1}{2}$&$\frac{2}{11}$&$\frac{1}{19}$&$\frac{132}{10411}$&$\frac{90}{34399}$& $\frac{15984}{33846961}$ \\ \addlinespace[3pt]
$i=2$&&0&$\frac{1}{2}$&$\frac{7}{11}$&$\frac{17}{38}$& $\frac{2344}{10411}$&$\frac{3086}{34399}$&$\frac{11121092}{372316571}$\\ \addlinespace[3pt]
$i=3$&&&0&$\frac{2}{11}$&$\frac{17}{38}$&$\frac{5459}{10411}$&$\frac{28047}{68798}$& $\frac{89321060}{372316571}$ \\ \addlinespace[3pt]
$i=4$&&&&0&$\frac{1}{19}$& $\frac{2344}{10411}$& $\frac{28047}{68798}$& $\frac{171080619}{372316571}$ \\ \addlinespace[3pt]
$i=5$&&&&&0&$\frac{132}{10411}$& $\frac{3086}{34399}$& $\frac{89321060}{372316571}$ \\ \addlinespace[3pt]
$i=6$&&&&&&0&$\frac{90}{34399}$&$\frac{11121092}{372316571}$ \\ \addlinespace[3pt]
$i=7$&&&&&&&0&$\frac{15984}{33846961}$\\\addlinespace[3pt]
$i=8$&&&&&&&&0 \\ \bottomrule

\end{tabular}
\end{center}
We can observe that any column in the table is symmetric; therefore, $H_d(z)$ is self-reciprocal. We prove it, for any $d \ge 0$, in the next section, and the fact plays a crucial role for the proof of Theorem \ref{main}.

\subsection{Symmetry of the $H$-polynomials}

In this section, we prove that the $H$-polynomials are self-reciprocal. 

For $d\ge 0$, define a matrix $\bm{F}_d$ by $\bm{F}_d =(f_{i,j})_{-1\le i,j \le d}$. For example, we have
$$ \bm{F}_0 =\begin{pmatrix}  1&0 \\0&1\end{pmatrix}   , \bm{F}_1 = \begin{pmatrix}  1&0&0  \\ 0&1  &1 \\0&0&2\end{pmatrix}  , \bm{F}_2 = \begin{pmatrix}  1&0&0 &0 \\ 0&1  &1&1 \\0&0&2 &6\\ 0&0&0&6\end{pmatrix}   .$$

\begin{lemma}\label{sym1}
For $d\ge 0$, the numbers $0!,1!,\dots, (d+1)!$ are the eigenvalues of $\bm{F}_d$, and the column vector $$ {}^t(F_{-1,d}, F_{0,d},\dots,F_{i,d},\dots,F_{d,d})$$ is an eigenvector for $(d+1)!$ of $\bm{F}_d$.
\begin{proof}
Since $\bm{F}_d$ is upper triangular and the numbers $0!,1!,\dots,(d+1)!$ are the diagonal entries, the first claim follows.

For the second claim, we have to show $$ \sum^d_{j=i} f_{i,j} F_{j,d} =(d+1)! F_{i,d}$$
for any $-1\le i \le d$. When $i=d$, it is clear. If $-1\le i\le d-1$, then \eqref{ind2} directly implies the equality. Hence, the result follows. 
\end{proof}
\end{lemma}

For $d\ge 1$ and a permutation $\sigma$ on the set $[d]=\{ 1,2,\dots,d \}$, define des$(\sigma)$ to be the number of $1\le i\le d-1$ such that $\sigma(i)>\sigma(i+1)$. For $0\le i \le d-1$ and $1\le j\le d$, we denote by $A(d,i,j)$ the number of permutations $\sigma$ on $[d]$ such that des$(\sigma)=i$ and $\sigma(1)=j$. In particular, $A(d,i,j)=0$ if $i\le -1$.

For $d\ge 0$, define a matrix $\bm{H}_d$ by $$ \bm{H}_d =\left( h^{(d)}_{i,j} \right)_{-1\le  i,j \le  d}= \big( A(d+2,i+1,j+2) \big)_{-1\le i,j \le d}.$$For example, we have $$ \bm{H}_0 =\begin{pmatrix}  1&0 \\0&1\end{pmatrix}   , \bm{H}_1 = \begin{pmatrix}  1&0&0  \\ 1&2  &1 \\0&0&1 \end{pmatrix}  , \bm{H}_2 = \begin{pmatrix}  1&0&0 &0 \\ 4&4 &2&1 \\ 1&2&4 &4 \\ 0&0&0&1 \end{pmatrix} , $$
$$  \bm{H}_3 = \begin{pmatrix}  1&0&0&0&0 \\ 11&8&4&2&1 \\ 11&14&16&14&11 \\1&2&4&8&11 \\0&0&0&0&1 \end{pmatrix}  , \bm{H}_4=\begin{pmatrix} 1&0&0&0&0&0 \\26&16&8&4&2&1 \\66&66&60&48&36&26 \\26&36&48&60&66&66 \\ 1&2&4&8&16&26 \\0&0&0&0&0&1 \end{pmatrix} .$$ 
It is easy to compute these examples by the following lemma:
\begin{lemma}[Lemma 2(i) of \cite{BW08}]\label{BW}
For $d \ge 0$ and $-1 \le i,j \le d,$ we have $$ h^{(d)}_{i,j}  = \sum^{j-1}_{\ell =-1}   h^{(d-1)}_{i-1,\ell}   + \sum^{d-1}_{\ell =j}   h^{(d-1)}_{i,\ell} .$$
\end{lemma}

The two matrices $\bm{F}_d$ and $\bm{H}_d$ have already been used in \cite{BW08}. In the paper, Brenti and Welker found that the matrices are similar (Lemma 4 (i) of \cite{BW08}); that is, there exists a nonsingular matrix $\bm{P}$ such that $\bm{P}^{-1} \bm{F}_d \bm{P} =\bm{H}_d$. In this paper, we explicitly describe the matrices $\bm{P}$ and $\bm{P}^{-1}$.

For $d\ge 0$, define the \textit{shift matrix }$\bm{S}_d$ by $$ \bm{S}_d =\left(  (-1)^{d+1+i+j}\binom{d-j}{i+1} \right)_{-1\le i,j \le d} .$$

\begin{lemma}\label{shift}
If $d\ge 0, f(z)=\sum^d_{i=-1} a_i z^{d-i}$, and $$h(z)=f(z-1)=\sum_{i=0}^{d+1} b_i z^{d+1-i} ,$$ then $$ b_{d-i} = \sum^d_{j=-1}  (-1)^{d+1+i+j} \binom{d-j}{i+1} a_j $$ for $-1\le i \le d$ and $$ \bm{S}_d  {}^t (a_{-1} ,a_0,\dots, a_d) =(b_{d+1}, b_d,\dots, b_0).$$
\begin{proof}
We have
\begin{align*}
h(z)=& \sum^d_{j=-1}  a_j (z-1)^{d-j}  \\
=&\sum^d_{j=-1}  a_j  \sum^{d-j}_{i=-1} \binom{d-j}{i} z^i (-1)^{d+i+j} \\
=&\sum^{d+1}_{i=0} \left( \sum^{d-i}_{j=-1} (-1)^{d+i+j} \binom{d-j}{i} a_j \right)z^i.
\end{align*}
Hence, $$b_{d-i} = \sum^d_{j=-1}  (-1)^{d+1+i+j}  \binom{d-j}{i+1}  a_j ,$$ and the first result follows. Since the $i$th coordinate,$-1  \le i \le d,$ of $\bm{S}_d {}^t (a_{-1},\dots, a_d)$ is $$\sum^d_{j=-1}  (-1)^{d+1+i+j}  \binom{d-j}{i+1}  a_j ,$$
the second result follows.
\end{proof}
\end{lemma}
\begin{lemma}\label{sym2}
For $d\ge 0$, the shift matrix is nonsingular, and the inverse matrix is given by $$\bm{S}_d'= \left(  \binom{j+1}{d-i} \right)_{-1\le i,j \le d} .$$
\begin{proof}
The $(i,j)$-entry of $\bm{S}_d \bm{S}_d'$ is $$ \sum_{-1\le k\le d} (-1)^{d+1+i+k}  \binom{d-k}{i+1} \binom{j+1}{d-k},$$
and we show that it is the Kronecker delta. By multiplying by $x^{i+1}$ and summing over $k\ge -1$, we have
\begin{align*}
&\sum^{\infty}_{i=-1} \sum_{-1 \le k\le d} (-1)^{d+1+i+k} \binom{d-k}{i+1} \binom{j+1}{d-k} x^{i+1}  \\
=& \sum^{\infty}_{k=-1}(-1)^{d+k} \binom{j+1}{d-k} \sum^{\infty}_{i=0}(-1)^i \binom{d-k}{i}x^i \\
=&\sum^{\infty}_{k=-1}(-1)^{d-k}\binom{j+1}{d-k} (1-x)^{d-k} \\
=&x^{j+1}.
\end{align*}
Hence, the result follows.
\end{proof}

\end{lemma}

\begin{lemma}\label{sym3}
For $d\ge 0$, we have $\bm{S}_d \bm{F}_d \bm{S}_d^{-1}  =\bm{H}_d$.
\begin{proof}
We prove the claim by induction on $d$. 

When $d=0$, it is clear.

Suppose that the equality holds for $d-1$. The $(i,j)$-entries of both sides are $$ \sum_{-1\le k,k'\le d}  (-1)^{d+1+i+k} \binom{d-k}{i+1} \binom{j+1}{d-k'}  f_{k,k'}$$
and $h^{(d)}_{i,j}$, respectively. By Lemma \ref{BW} and the assumption of induction, we have
\begin{align*}
h^{(d)}_{i,j} =&\sum^{j+1}_{\ell=1} A(d+1,i,\ell) + \sum^{d+2}_{\ell=j+3}A(d+1,i+1,\ell -1)  \\
=&\sum_{-1\le k,k'\le d-1} (-1)^{d+1+i+k} f_{k,k'}\Bigg( \sum^j_{\ell=0} \binom{d-1-k}{i} \binom{\ell}{d-1-k'} \\
& - \sum^d_{\ell=j+1} \binom{d-1-k}{i+1}\binom{\ell}{d-1-k'} \Bigg)\\
=&\sum_{-1\le k,k'\le d-1} (-1)^{d+1+i+k} f_{k,k'}\Bigg( \sum^j_{\ell=d-1-k'} \binom{d-1-k}{i} \binom{\ell}{d-1-k'} \\
&+  \sum^j_{\ell=d-1-k'} \binom{d-1-k}{i+1} \binom{\ell}{d-1-k'} - \sum^d_{\ell=d-1-k'} \binom{d-1-k}{i+1}\binom{\ell}{d-1-k'} \Bigg)\\
=&\sum_{-1\le k,k'\le d-1} (-1)^{d+1+i+k} f_{k,k'}\left(  \binom{d-k}{i+1}\binom{j+1}{d-k'} -\binom{d-1-k}{i+1}\binom{d+1}{k'+1} \right) .
\end{align*}
The last equality follows from the equality $\sum^b_{c=a}\binom{c}{a}=\binom{b+1}{a+1}$ for $a,b \ge 0$.

On the other hand, by \eqref{ind1}, we have
\begin{align*}
&\sum_{-1\le k,k' \le d} (-1)^{d+1+i+k} \binom{d-k}{i+1} \binom{j+1}{d-k'}f_{k,k'}\\
=&\sum_{-1\le k,k' \le d-1} (-1)^{d+1+i+k} \binom{d-k}{i+1} \binom{j+1}{d-k'}f_{k,k'} \\
&+\sum_{0\le k \le d} (-1)^{d+1+i+k} \binom{d-k}{i+1}  f_{k,d}\\
=&\sum_{-1\le k,k' \le d-1} (-1)^{d+1+i+k} \binom{d-k}{i+1} \binom{j+1}{d-k'}f_{k,k'} \\
& +\sum_{0\le k \le d} (-1)^{d+1+i+k} \binom{d-k}{i+1} \sum^d_{k'=k} \binom{d+1}{k'} f_{k-1,k'-1} \\
=&\sum_{-1\le k,k' \le d-1} (-1)^{d+1+i+k} \binom{d-k}{i+1} \binom{j+1}{d-k'}f_{k,k'} \\
&+ \sum_{-1\le k,k' \le d-1} (-1)^{d+i+k} \binom{d-1-k}{i+1} \binom{d+1}{k'+1}f_{k,k'} \\
=&\sum_{-1\le k,k'\le d-1} (-1)^{d+1+i+k} f_{k,k'}\left(  \binom{d-k}{i+1}\binom{j+1}{d-k'} -\binom{d-1-k}{i+1}\binom{d+1}{k'+1} \right) .
\end{align*}
Hence, the result follows.
\end{proof}
\end{lemma}

A square matrix $\bm{A}=(a_{i,j})_{1\le i,j \le n}$ over a ring is \textit{rotationally symmetric} if $a_{i,j}=a_{n+1-i,n+1-j}$ for any $1\le i,j \le n$. An eigenvector for a simple eigenvalue of a rotationally symmetric matrix has the following interesting property; that is, the eigenvector is almost symmetric.

\begin{lemma}\label{sym4}
Suppose that $\bm{A}=(a_{i,j})_{1\le i,j \le n}$ is a rotationally symmetric matrix over $\mathbb{C}$, $\lambda$ is a simple eigenvalue of $\bm{A}$, and $\bm{x}={}^t (x_1, x_2,\dots,x_n)$ is an eigenvector for $\lambda$. Then, we have $x_i= \delta x_{n+1-i}$ for any $1\le i \le n$, where $\delta=\pm 1$. Moreover, if the sum $\sum_i x_i$ is nonzero, then $\delta =1$.
\begin{proof}
It is easy to show that the vector $\bm{x}'={}^t (x_n,x_{n-1}, \dots,x_1)$ is also an eigenvector for $\lambda$. Since $\lambda$ is simple, the eigenspace for $\lambda$ is a one-dimensional space; therefore, there exists a complex number $\delta$ such that $\delta \bm{x}=\bm{x}'$. Since $\bm{x}$ is an eigenvector, $x_i$ is nonzero for some $i$. Furthermore, since $\delta x_i=x_{n+1-i}$ and $\delta x_{n+1-i} =x_i$, the constant $\delta$ must be $\pm 1$. Hence, the first claim follows.

Moreover, the equality $\delta \bm{x}=\bm{x}'$ implies $\delta\sum_i x_i=\sum_i x_i$. Since $\sum_i x_i$ is nonzero, the second claim follows.
\end{proof}
\end{lemma}

\begin{prop}\label{sym5}
For $d \ge 0$, the $H$-polynomial is self-reciprocal; that is, $H_{i,d} = H_{d+1-i,d}$ for any $0\le i \le d$.
\begin{proof}
By Lemma \ref{sym1} and \ref{sym3}, the matrices $\bm{F}_d$ and $\bm{H}_d$ have the same eigenvalues $0!, 1!,\dots, (d+1)!$ and the vector $$  \bm{S}_d {}^t (F_{-1,d}, F_{0,d}, \dots, F_{d,d}) =  {}^t ( H_{d+1,d} ,H_{d,d} ,\dots,  H_{0,d})$$ is an eigenvector for $(d+1)!$ of $\bm{H}_d$. Lemma 2 (ii) of \cite{BW08} implies that $\bm{H}_d$ is rotationally symmetric, and we have $$ \sum^{d+1}_{i=0}  H_{i,d} = H_d(1) =F_d(0) =1 \not =0.$$ Hence, Lemma \ref{sym4} completes this proof.
\end{proof}
\end{prop}

\subsection{Iterated barycentric subdivision}

In this section, we investigate the zeros of the $h$-polynomial of $\Delta$ under iterated barycentric subdivision.

As we have seen the table of $f_{i,d}$, barycentric subdivision generates many simplices; therefore, it is not easy to count the number of $i$-simplices of $\Delta^{(k)}$ precisely. However, if $k$ is sufficiently large, it is approximated by $f^{\Delta}_d F_{i,d} (d+1)!^k$.

\begin{lemma}\label{order}
For $0\le i \le d,$ we have $$f^{\Delta^{(k)}}_i=\sum_{j=0}^{d-i}  C_{j,i}(d+1-j)!^k  $$
for some rational numbers $C_{j,i}$. In particular, $C_{0,i}=f^{\Delta}_d F_{i,d}$. Hence, we have $$ f^{\Delta^{(k)}}_i = f^{\Delta}_d F_{i,d} (d+1)!^k +O \left( d!^k\right)$$ as $k\to \infty$.
\begin{proof}

We have the following recurrence: $$ \bm{F}_d  \begin{pmatrix}  f^{\Delta^{(k)}}_{-1} \\ f^{\Delta^{(k)}}_{0}  \\ \vdots \\ f^{\Delta^{(k)}}_{d}  \end{pmatrix} =\begin{pmatrix}  f^{\Delta^{(k+1)}}_{-1}  \\ f^{\Delta^{(k+1)}}_{0}  \\ \vdots \\ f^{\Delta^{(k+1)}}_{d} \end{pmatrix} .$$ Consider the generating function $$ M_i(x) = \sum^{\infty}_{k=0} f^{\Delta^{(k)}}_{i} x^k$$ for $0\le i \le d.$ We show that \begin{align} M_i(x) &= \sum^{d-i}_{j=0}  \frac{  C_{j,i}  }{ 1-(d+1-j)! x   }  \label{gen1}\end{align} for some rational numbers $C_{j,i}$ and,  in particular, \begin{align} C_{0,i} &= f^{\Delta}_d  F_{i,d}  \label{gen2}\end{align} by descending induction on $i$.

When $i=d$, multiply the recurrence $$f_{d,d} f_d^{\Delta^{(k)}} = (d+1)! f_d^{\Delta^{(k)}}  =  f_d^{\Delta^{(k+1)}} $$ by $x^k$ and sum over $k \ge 0$, and we have
\begin{align*}
(d+1)! M_d(x)&= \frac{  M_d(x) -f^{\Delta}_d }{ x  } \\
M_d&= \frac{ f^{\Delta}_d  }{  1-(d+1)!x   }.
\end{align*}
Since $F_{d,d}=1$, the claim follows.

Suppose that the claim is true for $i+1, i+2,\dots, d$. By multiplying the recurrence $$ f^{\Delta^{(k+1)}}_{i} = f_{i,i} f^{\Delta^{(k)}}_{i}   +f_{i,i+1}  f^{\Delta^{(k)}}_{i+1} +\dots +f_{i,d} f^{\Delta^{(k)}}_{d} $$ by $x^k$ and sum over $k \ge 0$, we have \begin{align*} M_i(x)  = \frac{ f^{\Delta}_i  }{ 1-(i+1)! x   }   +  \sum^d_{j=i+1} \sum^{d-j}_{\ell =0}  \frac{  f_{i,j} C_{\ell, j} x  }{  (1-(i+1)!x) (1-(d+1-\ell )!x)  } .\end{align*} By partial fraction decomposition, the first claim follows, and we have \begin{align*}  & \sum^d_{j=i+1}  \frac{  f_{i,j}  }{  1-(i+1)!x  }   \frac{  C_{0,j} x  }{ 1-(d+1)!x   }  \\  =&\sum^d_{j=i+1}  \frac{ f_{i,j}  C_{0,j}  }{  (d+1)!-(i+1)!    }  \left(  \frac{-1}{1-(i+1)!x)} + \frac{1}{  1-(d+1)!x  } \right) .\end{align*} By the assumption of induction and \eqref{ind2}, we have \begin{align*}  C_{0,i}&= \sum^d_{j=i+1}  \frac{f_{i,j} C_{0,j} }{   (d+1)!-(i+1)! }   \\ &=\frac{   f_d^{\Delta} }{ (d+1)!-(i+1)!   } \sum^d_{j=i+1} f_{i,j}  F_{j,d } \\ &=f_d^{\Delta} F_{i,d},\end{align*} and the claim follows.
\end{proof}
\end{lemma}

\begin{proof}[Proof of Theorem \ref{main}]

We have $$\begin{pmatrix} h^{\Delta^{(k)}}_{d+1} \\ \vdots \\ h^{\Delta^{(k)}}_{0}  \end{pmatrix} =\bm{S}_d \bm{F}_d^k\begin{pmatrix}    f^{\Delta}_{-1 }  \\ \vdots \\ f^{\Delta}_{0}  \end{pmatrix} =  \bm{S}_d\begin{pmatrix}   f^{\Delta^{(k)}}_{-1} \\ \vdots \\ f^{\Delta^{(k)}}_{0}  \end{pmatrix} .$$ Lemma \ref{shift} and \ref{order} imply
\begin{align*}
h^{\Delta^{(k)}}_{d-i}=& \sum^d_{j=-1} (-1)^{d+1+i+j} \binom{d-j}{i+1} f^{\Delta^{(k)}}_{j} \\
=&\sum^d_{j=-1} (-1)^{d+1+i+j} \binom{d-j}{i+1} \left( f^{\Delta}_d F_{j,d} (d+1)!^k + O\left( d!^k\right) \right) \\
=&f^{\Delta}_d H_{d-i,d}(d+1)!^k + O\left( d!^k\right)
\end{align*}
for $0\le i \le d .$ If $i=-1$, then $h_{d+1}^{\Delta^{(k)}} =(-1)^d \wchi(\Delta)$. Hence, we have 
\begin{align*}
h^{\Delta^{(k)}}(z)=& \sum_{i=0}^{d+1}  h^{\Delta^{(k)}}_{ d+1-i  } z^i\\
=& (-1)^d \wchi(\Delta) +\sum_{i=1}^{d+1}  h^{\Delta^{(k)}}_{ d+1-i  } z^i  \\
=& (-1)^d \wchi(\Delta) +  f_d^{\Delta}  (d+1)!^k\sum_{i=1}^{d+1}  H_{ d+1-i ,d } z^i  +\varepsilon_d(z) O\left( d!^k\right). \\
\end{align*}
Since $H_{0,d}=F_{-1,d}=0$, Proposition \ref{sym5} implies that $H_{d+1,d}$ is also zero. Hence, we have $$ h^{\Delta^{(k)}} (z) =(-1)^d \wchi(\Delta) +  f_d^{\Delta}  (d+1)!^k H_d(z)+  \varepsilon_d(z) O\left( d!^k\right). $$By the definitions of $H_d(z)$ and $F_{0,d}$, we have $H_{1,d} =F_{0,d} >0$. Since $H_d(z)$ is self-reciprocal by Proposition \ref{sym5}, we have $$ H_d(z) =H_{1,d} z \prod_n (z-\theta_n)^{e_n} ,$$
where $1+\sum_n e_n =d.$ 

Suppose that $R>0$ is sufficiently large such that the open ball $U(0: R)$ with the center zero of radius $R$ contains all zeros of $H_d(z)$. If $k$ is sufficiently large, then we have 
$$ |(d+1)!^k f^{\Delta}_d  H_d(z) | > | \wchi(\Delta)| +| \varepsilon(z) O\left( d!^k\right) | $$ on the circle $|s|=R$. Hence, Rouche's theorem implies that $h^{\Delta^{(k)}} (z)$ has exactly $d$ zeros in $U(0:R)$, however, the degree of $h^{\Delta^{(k)}} (z)$ is $d+1$; therefore, one of the zeros of $h^{\Delta^{(k)}} (z)$ lies in the exterior of $U(0:R)$. Hence, $| \rho^{\Delta^{(k)}}_{\infty}  |$ tends to $\infty$ as $k\to \infty$. Furthermore, since $h^{\Delta^{(k)}} (z)$ is a polynomial with real coefficients, the complex conjugation of $ \rho^{\Delta^{(k)}}_{\infty}  $ is also a zero of $h^{\Delta^{(k)}} (z)$, and they must coincide. Namely, $ \rho^{\Delta^{(k)}}_{\infty}  $ is real if $k$ is sufficiently large.

Suppose that $\varepsilon >0$ is sufficiently small such that the open ball $U(\theta_n : \varepsilon)$ does not intersect with $U(\theta_m : \varepsilon)$ if $n \not=m.$ If $k$ is sufficiently large, then the inequality above holds on the circle   $|s-\theta_n| =\varepsilon$. Hence, Rouche's theorem implies that $h^{\Delta^{(k)}} (z)$ has $e_n$ zeros in $U(\theta_n :\varepsilon)$. Hence, $e_n$ zeros of $h^{\Delta^{(k)}} (z)$ converge to $\theta_n$ as $k\to \infty$. Hence, we have $$ H_d(z)=   H_{1,d} \lim_{k\to \infty} \prod_{i=0}^{d-1} \left(z-\rho^{\Delta^{(k)}}_{i}  \right) . $$ Since $\prod_n \theta_n =(-1)^{d-1}$, the product $\prod^{d-1}_{i=1} \rho_i^{\Delta^{(k)}} $ converges to $(-1)^{d-1}$. Hence, the result follows.
\end{proof}

\begin{proof}[Proof of Theorem \ref{main2}]
Since any $h$-polynomial is monic, we have $$ h^{\Delta^{(k)}} (z)=\left(z-\rho^{\Delta^{(k)}}_{\infty} \right)\left(z-\rho^{\Delta^{(k)}}_{0} \right) \prod_{i=1}^{d-1} \left(z-\rho^{\Delta^{(k)}}_{i} \right) .$$
The constant term of $ h^{\Delta^{(k)}} (z)$ is $(-1)^d \wchi(\Delta)$, and the coefficient of $z^d$ of $ h^{\Delta^{(k)}} (z)$ is $ f^{\Delta^{(k)}}_0 -(d+1)$. Hence, we have 
\begin{align}
\wchi(\Delta)=&-\rho^{\Delta^{(k)}}_{\infty}\rho^{\Delta^{(k)}}_{0} \prod_{i=1^{d-1}} \rho^{\Delta^{(k)}}_{i} \label{yuru1} \\
(d+1)-f^{\Delta^{(k)}}_{0}=&\rho^{\Delta^{(k)}}_{\infty}+\rho^{\Delta^{(k)}}_{0}+\sum_{i=1}^{d-1} \rho^{\Delta^{(k)}}_{i} \label{yuru2}.
\end{align}
Since $\rho^{\Delta^{(k)}}_{0}$ and $\rho^{\Delta^{(k)}}_{i}$ converge and $d+1$ is fixed, \eqref{yuru2} implies $$\rho^{\Delta^{(k)}}_{\infty}\sim -f^{\Delta^{(k)}}_{0}$$ as $k\to \infty.$ By Lemma \ref{order}, we obtain $$ \rho^{\Delta^{(k)}}_{\infty} \sim - H_{1,d} f^{\Delta}_d (d+1)!^k.$$ Since $\prod^{d-1}_{i=1} \rho^{\Delta^{(k)}}_{i}$converges $(-1)^{d-1}$, we obtain the second result.
\end{proof}

\subsection{  The numbers $H_{1,d}$ }

In this section, we estimate the growth of $H_{1,d}$ as $d\to \infty$.

\begin{lemma}\label{2-power}
For $0 \le j \le d,$ we have $h^{(d)}_{0,j} =2^{d-j}.$
\begin{proof}
We prove the claim by induction on $d$.

If $d=0$, then $h^{(0)}_{0,0}=1.$ Hence, the claim is true.

Suppose that the claim is true for $d-1.$ By Lemma \ref{BW}, we have \begin{align*} h^{(d)}_{0,j} =&\sum^{j-1}_{\ell =-1} h^{(d-1)}_{-1,\ell}  + \sum^{d-1}_{\ell =j} h^{(d-1)}_{0,\ell} \\=&1+\sum^{d-1}_{\ell = j}2^{d-1-\ell} \\=&2^{d-j}. \end{align*} Hence, the result follows.
\end{proof}
\end{lemma}

\begin{lemma}\label{simply}
For $d \ge 1,$ we have \begin{multline*}  h^{(d)}_{0,d} \le h^{(d)}_{0,d-1} \le \dots \le h^{(d)}_{0,-1} \le h^{(d)}_{1,d}\le \dots \le h^{(d)}_{1,-1} \le \dots \le  h^{(d)}_{  \left[  \frac{d-1}{2}\right] , d} \le \dots \le  h^{(d)}_{  \left[  \frac{d-1}{2}\right] , N} , \end{multline*}where $$N= \begin{cases} -1 & \text{if $d$ is even} \\ \left[  \frac{d-1}{2}\right] & \text{if $d$ is odd.} \end{cases}$$
\begin{proof}
When $d=1$, the sequence is $h^{(1)}_{0,1} =1 \le 2 = h^{(1)}_{0,0}.$ Hence, the claim follows.

Assume that the claim is true for $d-1$ and $d$ is even. Then, the assumption of induction and Lemma \ref{BW} imply $h^{(d)}_{i-1,-1}=h^{(d)}_{i,d}$ for any $0\le i \le \left[  \frac{d-1}{2}\right]$ and 
$$  h^{(d)}_{i,j-1} -h^{(d)}_{i,j} =-h^{(d-1)}_{i-1,j-1} +h^{(d-1)}_{i,j-1} \ge -h^{(d-1)}_{i-1,-1}+h^{(d-1)}_{i,d-1} =0 $$ for any $0 \le j \le d$ (Lemma \ref{sym5} is used if $0 \le j \le \left[  \frac{d-1}{2}\right]$ ).

If $d$ is odd, for any $0 \le i <\left[  \frac{d-1}{2}\right]$ and $0\le j \le d,$ Lemma \ref{sym5} implies 
\begin{align*}
h^{(d)}_{i,j-1} -h^{(d)}_{i,j}=&-h^{(d-1)}_{i-1,j-1} + h^{(d-1)}_{i,j-1} \\
=&-h^{(d-1)}_{i-1,j-1} +h^{(d-1)}_{d-1-i-1, d-1-(j-1)-1}\\
=&-h^{(d-1)}_{d'-1,j-1} +h^{(d-1)}_{d'-1, 2d'-1}, 
\end{align*}
where $d=2d'+1.$ Since $j-1 \ge 2d' -j,$ we have $h^{(d)}_{i,j-1} \ge h^{(d)}_{i,j}$. Hence, the result follows.
\end{proof}
\end{lemma}

\begin{lemma}\label{det}
Let $$  M=\begin{pmatrix} -A_1&a_{1,2}&\cdots&a_{1,n} \\a_{2,1} &-A_2&&\vdots \\\vdots&& \ddots &\vdots \\a_{n,1} &\cdots&\cdots&-A_n \end{pmatrix}$$ be an $n\times n$ matrix such that $A_j$ and $a_{i,j}$ are positive real numbers for any $i$ and $j$. Suppose that $\sum^n_{i=1, i\not = j} a_{i,j} < A_j$ for any $j$.
\begin{enumerate}
\item The sign of the determinant of $M$ is $(-1)^n$. 
\item If we give positive real numbers $b_1, b_2,\dots, b_n$ and replace the $j$th column of $M$, $1\le j \le n,$ by ${}^t (-b_1, -b_2,\dots, -b_n)$, denote the matrix $M_j$, then the sign of the determinant of $M_j$ is also $(-1)^n$.
\end{enumerate}
\begin{proof}
We only give a proof of the first claim since the second can be proved similarly.

We prove it by induction on the size of the matrix.

If $n=1$, then we have $|M|=-A_1.$ Hence, the claim follows.

If we assume the truth of the claim for $n-1$, then we have
\begin{align*}
|M|=&\begin{vmatrix} -A_1 &a_{1,2}&a_{1,3}&\cdots&a_{1,n} \\ 0&-A_2+\frac{a_{2,1}}{A_1} a_{1,2}  & a_{2,3}+\frac{a_{2,1}}{A_1} a_{1,3}&\cdots&a_{2,n}+\frac{a_{2,1}}{A_1} a_{1,n} \\
0&a_{3,2}+\frac{a_{3,1}}{A_1} a_{1,2}&-A_3+\frac{a_{3,1}}{A_1} a_{1,3}&\cdots&a_{3,n}+\frac{a_{3,1}}{A_1} a_{1,n} \\
\vdots&\vdots&\vdots&\ddots&\vdots \\
0&a_{n,2}+ \frac{ a_{n,1} }{A_1}  a_{1,2}&a_{n,3} +\frac{ a_{n,1} }{A_1} a_{1,3}&\cdots&-A_n+\frac{ a_{n,1} }{A_1} a_{1,n}
\end{vmatrix} \\
=&-A_1\begin{vmatrix} a_{1,2}&a_{1,3}&\cdots&a_{1,n} \\ 
-A_2+\frac{a_{2,1}}{A_1} a_{1,2}  & a_{2,3}+\frac{a_{2,1}}{A_1} a_{1,3}&\cdots&a_{2,n}+\frac{a_{2,1}}{A_1} a_{1,n} \\
a_{3,2}+\frac{a_{3,1}}{A_1} a_{1,2}&-A_3+\frac{a_{3,1}}{A_1} a_{1,3}&\cdots&a_{3,n}+\frac{a_{3,1}}{A_1} a_{1,n} \\
\vdots&\vdots&\ddots&\vdots \\
a_{n,2}+ \frac{ a_{n,1} }{A_1}  a_{1,2}&a_{n,3} +\frac{ a_{n,1} }{A_1} a_{1,3}&\cdots&-A_n+\frac{ a_{n,1} }{A_1} a_{1,n}.
\end{vmatrix} 
\end{align*}
For any $2\le j\le n,$ the $j$th diagonal entry is negative and the sum of the $j$th column is 
$$-A_j+\sum^n_{i=2,i\not= j} a_{i,j} +\frac{a_{1,j}}{A_1}  \sum^n_{i=2} a_{i,1} <-A_j + \sum^n_{i=1, i\not = j} a_{a,j} <0. $$
Hence, the assumption of induction implies $\mathrm{sign} |M|= -1\times (-1)^{n-1} =(-1)^n,$ and the result follows.
\end{proof}
\end{lemma}

\begin{lemma}\label{positive}
For any $1 \le i \le d,$ the number $H_{i,d}$ is positive.
\begin{proof}
Since the column vector $\mathbb{H} ={}^t ( H_{0,d}, H_{1,d},\dots, H_{d+1,d})$ is an eigenvector for $(d+1)!$ of $\bm{H}_d$, we have $$ \left(   \bm{H}_d -(d+1)!E\right) \begin{pmatrix}  H_{0,d} \\\vdots\\H_{d+1,d}\end{pmatrix} =\begin{pmatrix}  0\\ \vdots \\0 \end{pmatrix}.$$ Since all the entries in the $-1$ and $d$th rows are zero except for $h^{(d)}_{-1,-1} =h^{(d)}_{d,d} =1,$ we have $H_{0,d} =H_{d+1,d} =0.$ Let $\bm{H}_d'=( h^{(d)}_{i,j} )_{    0\le i,j \le d-1}.$ Since the eigenvalue $(d+1)!$ is simple, the rank of $\bm{H}_d'-(d+1)! E$ is $d-1$. The sum of any column in the matrix is zero; therefore, the $(d-1)$th row can be removed. Let $\bm{H}_d'' =(h^{(d)}_{i,j})_{0\le i \le d-2, 0\le j \le d-1}$. We have $$  \left( \bm{H}_d'' -(d+1)! E \right) \begin{pmatrix}  H_{1,d} \\ \vdots \\ H_{d-1,d}\end{pmatrix} =-H_{d,d} \begin{pmatrix}  h^{(d)}_{0,d-1}  \\\vdots \\ h^{(d)}_{d-2,d-1}\end{pmatrix} .$$
If we regard $H_{1,d},\dots, H_{d-1,d}$ as variables and $H_{d,d}$ as a constant, this equation uniquely determines $H_{1,d}, \dots, H_{d-1,d}$. Since $H_{d,d} =F_{0,d}>0$, by Cramer's formula and Lemma \ref{det}, all $H_{i,d}$ are positive, and the result follows.
\end{proof}
\end{lemma}

\begin{prop}\label{N3}
For $d\ge 1$, we have $$ \frac{ \sqrt{2}^d  }{   (d+1)!d }  \le H_{1,d} \le \frac{    2^{d+1}   }{   (d+1)!   } .$$
\begin{proof}
Since ${}^t (H_{0,d} ,\dots, H_{d+1,d})$ is an eigenvector for $(d+1)!$ of $\bm{H}_d$, we have 
\begin{align*}
h^{(d)}_{0,0} H_{1,d}+h^{(d)}_{0,1}H_{2,d}\dots +h^{(d)}_{0,d-1} H_{d,d}=&(d+1)! H_{1,d} \\
h^{(d)}_{1,0} H_{1,d}+h^{(d)}_{1,1}H_{2,d}\dots +h^{(d)}_{1,d-1} H_{d,d}=&(d+1)! H_{2,d} \\
\vdots& \\
h^{(d)}_{d-1,0} H_{1,d}+h^{(d)}_{d-1,1}H_{2,d}\dots +h^{(d)}_{0,d-1} H_{d,d}=&(d+1)! H_{d,d}.
\end{align*}
Since $\sum^d_{i=1} H_{i,d}=1$ and all $H_{i,d}$ are positive (Lemma \ref{positive}), Lemma \ref{2-power} implies \begin{align*} (d+1)!H_{1,d} =& h^{(d)}_{0,0} H_{1,d} +\dots + h^{(d)}_{0,d-1} H_{d,d} \\\le& h^{(d)}_{0,0}  +\dots +h^{(d)}_{0,d-1} \\\le & 2^{d-1}.\end{align*}
 Hence, the right inequality follows.
 
By Lemma \ref{simply}, \ref{positive}, and Proposition \ref{sym5}, $H_{   \left[  \frac{d+1}{2}\right],d   }$ is the maximum among $H_{1,d},\dots, H_{d,d}$. Since $\sum^d_{i=1} H_{i,d}=1,$ we have $H_{ \left[  \frac{d+1}{2}\right],d} \ge \frac{1}{d}.$ Hence, Lemma \ref{2-power} implies $$ (d+1)!  H_{1,d}> h^{(d)}_{0,\left[  \frac{d+1}{2}\right]-1} H_{ \left[  \frac{d+1}{2}\right],d  }   \ge 2^{  d-\left[  \frac{d+1}{2}\right]+1  }  \frac{1}{d}  \ge \sqrt{2}^d \frac{1}{d}.$$
Hence, the result follows.
\end{proof}
\end{prop}

\section{The distribution of primes}

\subsection{Zeros of the $h$-polynomials of $\Delta_n$}

\begin{coro}\label{N5}
If $n \ge 6$, then $$ |\widetilde{\chi}(\Delta_n)|  \sim H_{1,d_n}  f^{\Delta_n}_{d_n}  \left|\rho^{\Delta_n^{(k_n)}}_0 \right|   (d_n+1)!^{k_n}$$ as $k_n \to \infty.$
\begin{proof}
It directly follows from Theorem \ref{main2}.
\end{proof}
\end{coro}

Hence, $|\widetilde{\chi}(\Delta_n)|$ is almost the right hand side. 

\subsection{The growth of the dimension of $\Delta_n$}
In this section, we estimate the growth of the dimension of $\Delta_n$.

\begin{prop}\label{dim}
We have $$ d_n=\dim \Delta_n= \frac{\log n}{ \log \log n}   +O\left(  \frac{\log n}{(\log \log n)^2}\right) .$$
\begin{proof}
By Theorem 4.7 of \cite{Apo}, we have $$ C_1 n\log n < p_n < C_2 n\log n  $$ for some constants $C_1, C_2,$ and any $n\ge 2$. For example, put $C_1=\frac{1}{6}$ and $C_2=24$. Then, $C_1< p_1=2<C_2$. By the definition, $\dim \Delta_n =d$ if and only if $$ p_1 p_2\dots p_{d+1} \le n < p_1 p_2 \dots p_{d+2}.$$ Hence, we have 
\begin{align}
C_1^{d+1} (d+1)! \prod^{d+1}_{m=2} \log m < n < C_2^{d+2} (d+2)! \prod^{d+2}_{m=2} \log m. \label{put1}
\end{align}
Since the function $\frac{\log x}{\log \log x}$ is simply increasing in $[e^e,\infty)$, we have 
$$ \frac{\log n}{\log \log n} < \frac{   (d+2) \log C_2 +\log (d+2)!  +\sum^{d+2}_{m=2} \log \log m     }{  \log \left( (d+2) \log C_2 +\log (d+2)!  +\sum^{d+2}_{m=2} \log \log m   \right)   }$$ if $n\ge e^e$. By integral test, we have
\begin{multline*}
d \log \log d - \Li d+C\le \sum^{d+1}_{m=2} \log \log m \le (d+1) \log \log (d+1) - \Li (d+1) +C,
\end{multline*}
where $\Li x=\int^x_2 \frac{dt}{\log t}$ and $C= \int^e_2 \frac{dt}{\log t} + \log \log 2$. By Stirling's formula, we have , for $\varepsilon >0,$ 
\begin{align}
\frac{\log n}{\log \log n }<& \frac{   (d+2)\log C_2 + (d+3)\log (d+3) + (d+2)\log \log (d+2)      }{    \log (d+3) +\log \log (d+3)     } \notag \\
<& (1+\varepsilon) d\label{put2}
\end{align}
if $n$ is sufficiently large. If we replace $\varepsilon$ by$$\frac{2\log C_2}{   \log d +\log \log d} ,$$ the inequality \eqref{put2} holds. Hence, we have
\begin{align*}
d>&\frac{1}{1+\varepsilon} \frac{\log n}{\log \log n}\\
=& \frac{\log n}{\log \log n} -\frac{\varepsilon}{1+\varepsilon}\frac{\log n}{\log \log n} \\
>&\frac{\log n}{\log \log n}-\frac{2\log C_2}{\log d} \frac{\log n}{\log \log n}.
\end{align*}
If we put $\varepsilon =\frac{1}{2}$ and take logarithm in \eqref{put2}, we have
\begin{align}
\log d > \log \left(   \frac{2}{3}  \frac{\log n}{\log \log n} \right)> \frac{1}{2} \log \log n. \label{put3}
\end{align}
Hence, we obtain $$ d> \frac{\log n}{\log \log n} -4\log C_2  \frac{\log n}{  (\log \log n)^2   }  .$$

By the left inequality of \eqref{put1}, we can similarly obtain $$ \frac{\log n}{\log \log n} \ge  (1-\varepsilon)d, $$ and we can replace $\varepsilon $ by $$  \frac{2\log C_1}{\log d +\log \log d} . $$By \eqref{put3}, we obtain $$  d\le \frac{\log n}{\log \log n}   +4 \log C_1 \frac{\log n}{ ( \log \log n )^2  }.$$ Hence, the result follows.
\end{proof}
\end{prop}

\subsection{The number of $d_n$-simplices of $\Delta_n$ }

In this section, we estimate the number of $d_n$-simplices of $\Delta_n$.

If $n=p_1^{e_1} p_2^{e_2}\dots p_d^{e_d}$, then the \textit{weight} of $n$ is defined by $\sum^d_{i=1} e_i$. Denote $\pi_d(x)$ the number of positive squarefree integers of weight $d$ not exceeding a real number $x$.

\begin{lemma}\label{N1}
We have $$  \pi_d(p_1p_2\dots p_{d+1} )  =O\left( C^d  \right)$$ for some constant $C>0$.
\begin{proof}
Suppose that a sequence of primes $q_1, q_2,\dots, q_d$ satisfies $q_1 q_2\dots q_d \le p_1 p_2 \dots p_{d+1}$ and $q_1<\dots <q_d$. We count the number of such $(q_1,\dots, q_d)$. The greatest member $q_d$ does not exceed $p_{d^3}$, since $q_1 q_2\dots q_d \ge p_1p_2 \dots p_{d-1} q_d$; that is, $$    q_d \le p_d p_{d+1}  \le C_2^2 (d+1)^2 \log^2 (d+1) . $$ Hence, we choose $q_1,q_2,\dots, q_d$ from the set $\{  p_1,p_2,\dots,p_{d^3}    \} .$

We choose $m$ satisfying $m \ge \frac{e^2 C_2}{C_1}$.  The number of $i, 1\le i \le d,$ such that $q_i \ge p_{md}$ is smaller than $[\log d]$. Indeed, if $ q_1,q_2,\dots, q_{d-[\log d]}   <p_{md} $ and $ q_{d-[\log d] +1} , \dots, q_d \ge p_{md}$, then we have 
$$  q_1q_2 \dots q_d \ge p_1 p_2 \dots p_{d-[\log d]}  p_{md}p_{md+1}  \dots p_{md+[\log d]} .$$ Hence, it must be that $$p_{md}\dots p_{md+[\log d]} \le p_{d-[\log d] +1}  \dots p_{d+1} .$$ The left hand side is, at least, $$  p_{md} \dots   p_{md+[\log d]} \ge \left(  C_1 md \log d  \right)^{[\log d]} ,$$ and the right hand side is, at most, $$ p_{d-[\log d] +1} \dots p_{d+1} \le \left(  C_2 d \log d  \right)^{[\log d]+1}  . $$ However, the condition $m\ge \frac{e^2 C_2}{C_1}$ implies $$ p_{md} \dots p_{md+[\log d]} > p_{d-[\log d]+1} \dots p_{d+1} .$$ Hence, the claim follows.

By Stirling's formula, we have 
\begin{align*}
\pi_d (p_1\dots p_{d+1}) \ll &\binom{md}{d} \frac{d^{3[\log d]}}{[\log d]!} \\
\ll & \frac{m^d d^d}{d!} d^{3\log d} \\
\ll & \frac{e^d m^d d^d}{d^d}d^{3\log d}\\
\ll &C^d
\end{align*}
for some constant $C>0$. Hence, the result follows.
\end{proof}
\end{lemma}

\begin{prop}\label{N2}
We have $$ f^{\Delta_n}_{d_n} =O\left(  C^{d_n} \right)$$ for some constant $C>0$.
\begin{proof}
By the definition of $\Delta_n$, we have $$ f^{\Delta_n}_{d_n}  =\pi_{d_n +1}(n) \le \pi_{d_n+1}(p_1p_2\cdots p_{d_n+2}). $$
The result follows from Lemma \ref{N1}.
\end{proof}
\end{prop}

\subsection{The special zeros $\rho^{\Delta_n^{(k_n)}}_0  $  }

The last problem is to estimate $ \left|\rho^{\Delta_n^{(k_n)}}_0 \right| $.

\begin{defi}
By Corollary \ref{N5}, we can put $$ \left| \rho^{\Delta_n^{(k_n)}}_0   \right|\sim \alpha_n(d_n+1)!^{k_n} $$ as $k_n \to \infty$, since $\widetilde{\chi}(\Delta_n), H_{1,d_n}$, and $ f^{\Delta_n}_{d_n} $ are constant for any fixed $n\ge 6$. Namely,  $$  \alpha_n = \frac{   \wchi(\Delta_n)  }{  H_{1,d_n} f^{\Delta_n}_{d_n}      } .$$
\end{defi}

The numbers $\alpha_n$ are mysterious at present, however it is essential in this paper. At the beginning, our approach is topological. We use the two elementary topological notions: Euler characteristic and barycentric subdivision. However, the results of Proposition \ref{P1} and \ref{RHE} are very strong. I guess that $\alpha_n$ are related to something topological deeply, and I require the something to be irregular, difficult, and a major object in topology since prime numbers and $M(x)$ are so in number theory. For example, the stable homotopy groups of spheres satisfy the requirements. At least, the Riemann zeta function and the stable homotopy groups of spheres are related via the Bernoulli numbers. See Theorem 1.5 and 1.6 of \cite{Ada65}. However, to estimate $\alpha_n$ remains as an open problem.

We end this paper by proving  Proposition \ref{P1} and \ref{RHE}.

\begin{proof}[Proof of Proposition \ref{P1}]
By Corollary \ref{N5} and Proposition \ref{N3} and \ref{N2}, we have
\begin{align*}
\widetilde{\chi}(\Delta_n)  \ll& H_{1,d_n} f^{\Delta_n}_{d_n}  \left| \rho^{\Delta_n^{(k_n)}}_0 \right| (d_n+1)!^{k_n} \\
\ll& \frac{C^{d_n}}{(d_n+1)!} |\alpha_n | \\
\ll& C^{d_n} (d_n+1)!\\
\ll& d_n^{d_n}C^{d_n} \\
=&\exp\left(  d_n \log d_n  +d_n \log C  \right).
\end{align*}
By Proposition \ref{dim}, we have
\begin{align*}
\widetilde{\chi}(\Delta_n)  \ll &  \exp \left(  \frac{\log n}{\log \log n} \log \left(  \frac{\log n}{\log \log n} \right) +A\frac{\log n}{\log \log n}   \right) \\
=&\exp \left(  \log n -\frac{\log n}{\log \log n} \log \log \log n +A\frac{\log n}{\log \log n}   \right)\\
\ll& n \exp \left(  -A \frac{\log n \log \log \log n}{\log \log n}   \right)
\end{align*}
for some constant $A>0$. Hence, the result follows.
\end{proof}

\begin{proof}[Proof of Proposition \ref{RHE}]
If $$\alpha_n= O\left(   (d_n+1)!^{ \frac{3}{2} +\varepsilon }\right) ,$$
then we have 
\begin{align*}
\wchi(\Delta_n) \ll & \frac{C^{d_n}}{(d_n+1)!} |\alpha_n |\\
\ll & (d_n+1)!^{ \frac{1}{2} +\varepsilon } C^{d_n} \\
\ll & n^{ \frac{1}{2} +\varepsilon }.
\end{align*}

Next, we assume the Riemann Hypothesis. Since $n!> (ne^{-1})n$, we have
\begin{align*}
(d_n+1)!  \gg& (d_n e^{-1})^{d_n }  \\
=&\exp ( -d_n +d_n\log d_n) \\
\gg & n \exp \left(  -A \frac{ \log n \log \log \log n }{ \log \log n }  \right)
\end{align*}
by Proposition \ref{dim}. Hence, we have
\begin{align*}
|\alpha_n| \ll& \frac{ | \wchi(\Delta_n)| }{ H_{1,d_n}  f_{d_n}^{\Delta_n} }  \\
\ll & \frac{ (d_n+1)!  n^{\frac{1}{2} +\varepsilon}  }{ \sqrt{2}^{d_n} \cdot 1  }  \\
\ll & (d_n+1)! \sqrt{2}^{-d_n} (d_n+1)!^{ \frac{1}{2} +\varepsilon }  \exp \left(  B \frac{ \log n \log \log \log n }{ \log \log n }  \right) \\
\ll& (d_n+1)!^{ \frac{3}{2} +\varepsilon  } \exp \left(  C \frac{ \log n \log \log \log n }{ \log \log n }  \right) \\
\ll& (d_n+1)!^{ \frac{3}{2} +\varepsilon'  }.
\end{align*}
Hence, the result follows.
\end{proof}

Note that Proposition \ref{dim} is precise and it is meaningless to improve Proposition \ref{N3} and Lemma \ref{N1}. If we supposed $$ H_{1,d} =\frac{  \sqrt{2}  }{  (d+1)! d  } $$ and $$\pi_d(p_1p_2\dots p_{d+1})=1 ,$$ we would have $\widetilde{\chi}(\Delta_n) \ll \frac{     \sqrt{2}^{d_n}  }{ (d_n +1)!  }  |\alpha_n|$, and this is almost $\frac{     C^{d_n}  }{ (d_n +1)!  }  |\alpha_n|$.


\begin{thebibliography}{99}

\bibitem{Ada65} J. F. Adams. On the groups $J(X)$-IV. \textit{Topology}, 5, 21--71, 1965.

\bibitem{Apo}  T. M. Apostol. \textit{Introduction to Analytic Number Theory,} Springer, 1976.

\bibitem{BL08}  C. Berger and T. Leinster. The Euler characteristic of a category as the sum of a divergent series, \textit{Homology, Homotopy Appl., }10(1): 41-51, 2008.

\bibitem{Bjo11} A. Bj\"orner. A cell complex in number theory. \textit{Advances in Applied Math.,} 46: 71--85, 2011.



\bibitem{BW08}  F. Brenti and V. Welker. $f$-Vectors of barycentric subdivisions. \textit{Math. Z.,} 259: 849--865, 2008.
     
\bibitem{DPS12} E. Delucchi, A. Pixton, and L. Sabalka. Face vectors of subdivided simplicial complexes, \textit{Discrete Math.,} 312: 248--257, 2012.
     
\bibitem{Ivic}  A. Ivi\'{c}. \textit{The Riemann zeta-function. Theory and applications}, Dover Publications, 2003. 

\bibitem{Koz08} D. Kozlov. \textit{Combinatorial Algebraic Topology}. Springer-Verlag, 2008.




\bibitem{OR85} A. M. Odlyzko and H. J. J. te Riele. Disproof of the Mertens conjecture, \textit{J. Reine Angew. Math.,} 357: 138--160, 1985.

\bibitem{Spa66} E. Spanier. \textit{Algebraic topology}. McGraw-Hill, 1966. 

\bibitem{Sta97} R. P. Stanley. \textit{Enumerative Combinatorics Vol.1}. Cambridge University Press, Cambridge, 1997.

\bibitem{Tit86} E. C. Titchmarsh. \textit{The Theory of the Riemann Zeta-function.} Oxford, 1951 (2nd ed., revised by D. R. Heath-Brown, Oxford, 1986).
\end{thebibliography}
\end{document}